\documentclass{amsart}
\usepackage{euscript,amsmath, amssymb, amsfonts, mathtools}
\usepackage{mathrsfs}
\usepackage{mathabx}
\usepackage{stmaryrd}
\usepackage{crossreftools}

\makeatletter
\@namedef{subjclassname@1991}{Subject}
\makeatother

\usepackage[colorlinks=true, allcolors=blue]{hyperref}

\usepackage{tikz-cd}

\numberwithin{equation}{section}
\usepackage{enumitem}   

\newtheorem{theorem}{Theorem}[section]
\newtheorem{lemma}[theorem]{Lemma}
\newtheorem{proposition}[theorem]{Proposition}
\newtheorem{corollary}[theorem]{Corollary}

\theoremstyle{definition}
\newtheorem{definition}[theorem]{Definition}

\theoremstyle{remark}
\newtheorem*{remark}{Remark}

\usepackage{graphicx}  
\usepackage{icomma}
\usepackage{bbm}
\usepackage{tikz}
\usetikzlibrary{shapes}
\usepackage{color}
\usepackage{algorithmic}
\usepackage{verbatim}
\usepackage{xurl}
\usepackage{indentfirst}
\usepackage{lipsum}
\usepackage[backend = biber]{biblatex}
\usepackage{placeins}
\usepackage{bigints}
\usepackage{subcaption}
\usepackage{appendix}

\DeclareMathOperator{\Jac}{Jac}

\DeclareMathOperator{\CC}{\mathbb C}

\DeclareMathOperator{\ZZ}{\mathbb Z}

\DeclareMathOperator{\PP}{\mathbb P}

\DeclareMathOperator{\Per}{Per}

\DeclareMathOperator{\Res}{\mathcal R}
\DeclareMathOperator{\Discr}{\mathcal D}

\usepackage{csquotes}

\begin{document}
	\title[An explicit expression of the Richelot isogeny]{An explicit expression of the Richelot isogeny through Kleinian hyperellyptic functions}
	
	\author{Matvey Smirnov}
	\address{119991 Russia, Moscow GSP-1, ul. Gubkina 8,
		Institute for Numerical Mathematics,
		Russian Academy of Sciences}
	\email{matsmir98@gmail.com}

	\begin{abstract}
		We find an explicit expression for the Richelot isogeny of Kummer surfaces of genus 2 curves in terms of Kleinian hyperelliptic functions of weight 2. We use this expression to relate Kleinian hyperelliptic functions associated to Richelot isogenous curves.
		
	\end{abstract}
	\maketitle
\section{Introduction}
	
	This paper aims to derive a part of an algorithm that computes Kleinian hyperelliptic functions of genus 2 associated with a given curve $C$. We refer to~\cite{KleinianWeight2} for a more thorough discussion of how to construct such an algorithm. Briefly, we aim to imitate for curves of genus 2 the scheme of the classical Landen's algorithm. That is, we consider a sequence of curves $C_1 = C,C_2, \dots, C_n$ of genus $2$ with a given isogeny $\Jac(C_{k+1}) \to \Jac(C_n)$. If we are able to express the Kleinian hyperelliptic functions associated with $C_k$ through the same set of functions associated with $C_{k+1}$, then we can recursively reduce the problem to the curve $C_n$. If we appropriately choose isogenous curves, then we can hope that Kleinian functions associated with $C_n$ can be approximated via some other method.
	
	For the construction of the isogenous curves we propose to use the well-known construction due to Richelot (see, e.g.~\cite{Humbert} and~\cite{BostMestre}). Note that there is a choice of (in general) 15 distinct possibilities for $C_{k+1}$ on each step. The problem of the choice of $C_{k+1}$ on each step and the approximation of the resulting Kleinian functions will be the subject of a forthcoming paper. In this paper we solve the problem of expressing Kleinian functions associated with a given curve $C$ through the Kleinian functions associated with a Richelot isogenous curve $\hat C$. 
	
	For our purposes it is customary to find the mentioned expression for Kleinian fuctions of weight 2 introduced in~\cite{KleinianWeight2}, instead of classical Kleinian functions~\cite{Buch}. One of the reasons for that is the fact that the definition of these functions does not require any conditions on the algebraic equation of the curve (in contrast to $\sigma$-function, which is usually defined only for hyperelliptic curves with Weierstrass point at infinity). Moreover, since the role of Kleinian functions of weight $2$ is similar to that of theta-functions of weight $2$, they define an embedding of the Kummer surface of a genus $2$ curve into $\CC\PP(3)$. This allows us to find an equation that connects Kleinian functions of weight $2$ for Richelot isogenous curves just by finding an explicit formula for the Richelot isogeny of Kummer surfaces in the coordinates defined by Kleinian functions of weight $2$. Similar formula (in purely algebraic setting) was obtained in~\cite[Section~9]{CasselsFlynn}. Our approach differs from that in~\cite{CasselsFlynn} as we heavily rely on transcendental methods (and as a final result we obtain a relation for transcendental functions -- Kleinian functions of weight $2$).
	
	The paper is organized as follows. In Section~\ref{secPrelim} we collect all necessary known facts about curves of genus 2 and Kleinian functions, Richelot isogenies, and Kummer surfaces. In Section~\ref{secRichelotIsogeny} we derive a general form of the Richelot isogeny for Kummer surfaces and collect several facts about it, so we are able to calculate all parts of the formula in Section~\ref{secCalculations}. 
	
	\section{Preliminaries}\label{secPrelim}
	In this section we collect all necessary known facts that we shall use to derive the explicit formula for the Richelot isogeny. In particular, we recall the notation for special functions introduced in~\cite{KleinianWeight2}, several facts about Richelot isogenies (we also include some proofs for the facts that are not covered in literature in the form that we shall use), and nomenclature of nodes and tropes of Kummer surfaces.
	\subsection{Curves of genus 2 and Kleinian hyperelliptic functions} As in~\cite{KleinianWeight2} we shall denote the space of complex polynomials of degree $\le k$ by $\mathfrak P_k$. Given a polynomial $f$ we denote by $f_j$ the coefficient of $x^j$ in $f$. Moreover, we call a polynomial $f \in \mathfrak P_6$ {\it{admissible}} if it does not have multiple roots and $\deg f$ is equal to either $5$, or $6$. With an admissible polynomial we associate a genus 2 curve $\mathcal X_f$ (whose affine part is defined by the equation $y^2 = f(x)$). On this curve we fix the $1$-forms
	\[
		\omega_1^f = \frac{dx}{y},\;\;\omega_2^f = \frac{xdx}{y},\;\;r_1^f = \frac{f_3x + 2f_4x^2 + 3f_5x^3 + 4f_6x^4}{4y}dx,\;\;r_2^f = \frac{f_5x^2 + 2f_6x^3}{4y}dx.
	\]
	The forms $\omega_1^f, \omega_2^f$ constitute a basis of the space of holomorphic $1$-forms on $\mathcal X_f$. We denote the lattice of periods of forms $\omega_1^f, \omega_2^f$ by $\Per_f \subset \CC^2$. This lattice is canonically equipped with the intersection pairing from $H_1(\mathcal X_f, \ZZ)$. Moreover, given a period $w \in \Per_f$ we define $\eta^f(w) \in \CC^2$ by the rule
	\begin{equation*}
		\eta^f(w) = -\int_\gamma \begin{pmatrix} r_1^f \\ r_2^f\end{pmatrix},\text{ where } w = \int_\gamma \begin{pmatrix} \omega_1^f \\ \omega_2^f \end{pmatrix}.
	\end{equation*}
	The periods $\eta^f$ are well-defined, since $\rho_1^f$ and $\rho_2^f$ are of the second kind. We shall use several times the equalities
	\begin{equation}\label{eqLegendre}
		\begin{gathered}
			\eta_A^T B - A^T \eta_B = B \eta_A^T - A\eta_B^T= 2i\pi I, \\
			\left(\eta_A \eta_B^T\right)^T = \eta_A\eta_B^T,\;\; (\eta_A^TA)^T = \eta_A^TA,\;\;(\eta_B^TB)^T = \eta_B^TB,
		\end{gathered}
	\end{equation}
	which constitute the analog of the famous Legendre's identities (see~\cite[eq.~(3.4)]{KleinianWeight2}).
	
	By $\Jac_f$ we denote the quotient $\CC^2/\Per_f$, which is clearly isomoprhic to the Jacobian variety of $\mathcal X_f$. Given a degree $2$ divisor $D$ we denote by $\mathscr A_f(D)$ the image of the divisor class $D - \mathfrak L_f$ in $\Jac_f$ under the Abel map, where $\mathfrak L_f$ is the canonical hyperelliptic divisor class on $\mathcal X_f$. We canonically identify the symmetric square $\mathcal X_f^{(2)}$ with the variety of effective degree $2$ divisors on $\mathcal X_f$, so $\mathcal A_f:\mathcal X_f^{(2)} \to \Jac_f$ is a bimeromorphic map.
	
	Finally, we recall the definition of special functions discussed in~\cite{KleinianWeight2}. The functions $\xi^f_{11}$, $\xi_{12}^f$, and $\xi_{22}^f$ on the surface $\mathcal X_f^{(2)}$ are defined by the formulas
	\begin{equation}\label{eqXiDef}
		\xi^f_{22}(D) = x_1 + x_2,\;\;\xi^f_{12}(D) = -x_1x_2,\;\;\xi^f_{11}(D) = \frac{F_f(x_1,x_2) - 2y_1y_2}{4(x_1 - x_2)^2},
	\end{equation}
	where $D = (x_1,y_1) + (x_2,y_2) \in \mathcal X_f^{(2)}$ and 
	\begin{equation}\label{eqFdef}
		F_f(a,b) = \\ 2f_0 + f_1(a + b) + 2f_2ab + f_3ab(a + b) + 2f_4a^2b^2 + f_5a^2b^2(a + b) + 2f_6a^3b^3.
	\end{equation}
	Using the fact that $\mathcal A_f$ is bimeromorphic it is possible to introduce meromorphic functions $\wp^f_{11}$, $\wp^f_{12}$, and $\wp^f_{22}$ on $\Jac_f$ that satisfy the equations $\wp^f_{jk}(\mathscr A_f(D)) = \xi^f_{jk}(D)$ for a generic $D \in \mathcal X_f^{(2)}$. Moreover, we recall~\cite{KleinianWeight2} the definition of the space $\mathfrak S_f$ of {\it{Kleinian hyperelliptic functions of weight 2}}. By definition it consists of all holomorphic functions $\phi$ on $\mathbb C^2$ that satisfy the property
	\begin{equation}\label{eqSpaceDef}
		\phi(z + w) = \exp\left[2\eta^f(w)^T\left(z + \frac{w}{2}\right)\right]\phi(z)
	\end{equation}
	for all $z \in \CC^2$ and $w \in \Per_f$. It is proved (see~\cite[Corollary~3.5]{KleinianWeight2}) that this space is spanned by the four functions $S^f$, $S^f_{11}$, $S^f_{12}$, $S^f_{22}$, where $S^f$ is the unique element of $\mathfrak S_f$ that satisfies $S^f(z) = z_1^2 + \bar{o}(z^2)$, and $S_{jk}^f = S^f \wp_{jk}^f$. 
	
	In this paper we shall use the order 2 Taylor expansions
	\begin{equation}\label{eqTaylorExpansions}
		S_{22}^f(z) = 2z_1z_2 + \bar{o}(z^2),\;\;
		S_{12}^f(z) = -z_2^2 + \bar{o}(z^2),\;\;
		S_{11}^f(z) = 1 + \bar{o}(z^2)
	\end{equation}
	of functions $S_{jk}^f$ (see~\cite[Theorem~4.1]{KleinianWeight2}), and the fact that the space $\mathfrak S_f$ is canonically isomorphic to the space of theta-functions of weight 2. That is (see~\cite[Definition~II.1.2]{mumfordI}), given a Riemann matrix $\Omega$ we consider the space $R_2^\Omega$ that consists of all holomorphic functions $\phi$ on $\CC^2$ such that 
	\[
	\phi(z + n + \Omega m) = \exp\left[-2i\pi m^T \Omega m -4i\pi m^T z\right]\phi(z)
	\]
	for all $z \in \CC^2$ and all $m,n \in \ZZ^2$. If $a_1,a_2,b_1,b_2 \in \Per_f$ is a symplectic basis, then we can consider matrices
	\[
		A = \begin{pmatrix}
			a_1 & a_2
		\end{pmatrix},\;\;
		B = \begin{pmatrix}
			b_1 & b_2
		\end{pmatrix},\;\;
		\eta_A = \begin{pmatrix}
			\eta^f(a_1) & \eta^f(a_2)
		\end{pmatrix},\;\;
		\eta_B = \begin{pmatrix}
			\eta^f(b_1) & \eta^f(b_2)
		\end{pmatrix}.
	\]
	Then the matrix $\Omega = A^{-1}B$ is a Riemann matrix and the transformation $T$ defined by the rule
	\begin{equation}\label{eqTisomorphism}
		T(\phi)(z) = \exp\left[z^T (\eta_AA^{-1})z\right]\phi(A^{-1}z)
	\end{equation}
	is an isomorphism of the space $R^\Omega_2$ onto $\mathfrak S_f$ (see~\cite[Proposition~3.1]{KleinianWeight2}).
	\subsection{Richelot construction}
	Let us recall the definition and properties of the Richelot isogeny. The construction we present below is essentially due to Humbert~\cite{Humbert}. An elaborated exposition can be found in Smith~\cite[Chapter~8]{Smith}. 
	
	At first we establish several constructions with polynomials of degree $\le 2$. If $p,q \in \mathfrak P_2$, then the polynomial $[p,q] = p'q - q'p$ also belongs to $\mathfrak P_2$. It can be verified that the operation $[\cdot,\cdot]$ turns the space $\mathfrak P_2$ into a Lie algebra. We use symbols $\Res$ and $\Discr$ to denote the {\it{resultant}} and the {\it{discriminant}} of elements of $\mathfrak P_2$ respectively, that is
	\begin{equation}\label{eqDiscrAndResDef}
		\Discr(p) = p_1^2 - 4p_0 p_2,\;\; \Res(p,q) = (p_2 q_0 - p_0 q_2)^2 + (p_2 q_1 - p_1 q_2)(p_0 q_1 - p_1 q_0).
	\end{equation}
	Note that we apply the formulas~\eqref{eqDiscrAndResDef} for discriminant and resultant to all elements of $\mathfrak P_2$, even if the degree of the polynomials is less than $2$. For polynomials $p(x) = a(x - t_1)(x - t_2)$ and $q = b(x - s_1)(x - s_2)$ of degree exactly $2$ it is known that
	\begin{equation}\label{eqDiscrAndResRootFormula}
		\Discr(p) = a^2(t_1 - t_2)^2,\;\;\Res(p,q) = a^2b^2(t_1 - s_1)(t_1 - s_2)(t_2 - s_1)(t_2 - s_2).
	\end{equation}
	Finally, it will be customary to us to consider point at infinity in the Riemann sphere $\CC\PP(1)$ to be a simple root (resp. double root) of non-zero polynomial $p \in \mathfrak P_2$, if $\deg p = 1$ (resp. $\deg p = 0$). With this convention all non-zero elements in $\mathfrak P_2$ have exactly two roots counting multiplicity. 
	
	Below we state the basic properties of the introduced constructions.
	
	\begin{proposition}\label{propBasicBasic}
		The following statements hold for all non-zero $p,q \in \mathfrak P_2$.
		\begin{enumerate}[label=(\roman*)]
			\item\label{BasicBasici} $\Discr(p) = 0$ if and only if $p$ has only one root (of multiplicity two).
			\item\label{BasicBasicii} $\Res(p,q) = 0$ if and only if $p$ and $q$ have a common root.
			\item\label{BasicBasiciii} $[p,q] = 0$ if and only if $p$ and $q$ are proportional.
			\item\label{BasicBasiciv} $\Discr([p,q]) = 4\Res(p,q)$.
			\item\label{BasicBasicv} Let $S$ be a M\"obius transformation, i.e.
			\[
			S(x) = \frac{ax + b}{cx + d}, \;\;\det \begin{pmatrix} a & b \\ c & d \end{pmatrix} = 1.
			\]
			Then $[(cx + d)^2p \circ S, (cx + d)^2q \circ S] = (cx + d)^2[p,q] \circ S$.
		\end{enumerate}
	\end{proposition}
	\begin{proof}
		The statements~\ref{BasicBasiciii},~\ref{BasicBasiciv}, and~\ref{BasicBasicv} can be verified by a simple calculation. The statements~\ref{BasicBasici} and~\ref{BasicBasicii} follow immediately from~\eqref{eqDiscrAndResRootFormula}, if $\deg p = \deg q = 2$. For the remaining cases the proof is an easy calculation.
	\end{proof}

	Now we define yet another function $\Delta$ of three arguments $p,q,r \in \mathfrak P_2$. That is, $\Delta(p,q,r)$ is the determinant of the matrix
	\[
	\begin{pmatrix}
		p_0 & q_0 & r_0 \\
		p_1 & q_1 & r_1 \\
		p_2 & q_2 & r_2
	\end{pmatrix},
	\]
	where $p(x) = p_0 + p_1x + p_2x^2$, etc.
	
	\begin{proposition}\label{propBasicProperties}
		Assume that $p,q,r \in \mathfrak P_2$ and let $\hat p = [q,r]$, $\hat q = [r,p]$, and $\hat r = [p,q]$. Then the following statements hold.
		\begin{enumerate}[label=(\roman*)]
			\item\label{BasicPropertiesi} $[\hat p, \hat q] = -2\Delta(p,q,r)r$, $[\hat q, \hat r] = -2\Delta(p,q,r)p$, and $[\hat r, \hat p] = -2\Delta(p,q,r)q$.
			\item\label{BasicPropertiesii} $\Delta(\hat p, \hat q, \hat r) = -2\Delta(p,q,r)^2$. 
			\item\label{BasicPropertiesiii} $\Res(\hat p, \hat q) = \Delta(p,q,r)^2\Discr(r)$, $\Res(\hat p, \hat r) = \Delta(p,q,r)^2\Discr(q)$, and $\Res(\hat q, \hat r) = \Delta(p,q,r)^2\Discr(p)$.
			\item\label{BasicPropertiesiv} Let $S$ be a M\"obius transformation as in Proposition~\ref{propBasicBasic}~\ref{BasicBasicv}. Then
			\[
			\Delta( (cx + d)^2p \circ S, (cx + d)^2q \circ S, (cx + d)^2r \circ S) = \Delta(p,q,r).     
			\]
		\end{enumerate}
	\end{proposition}
	\begin{proof}
		Direct verification.
	\end{proof}
	
	Let $f$ be an admissible polynomial and assume that $f = pqr$ for some $p,q,r \in \mathfrak P_2$. If $\Delta(p,q,r) \ne 0$, then we can consider the polynomial 
	\begin{equation}\label{eqHatFDef}
		\hat {f} = \frac{1}{4\Delta(p,q,r)}[q,r] [r,p] [p,q].
	\end{equation}
	It is clear from Proposition~\ref{propBasicBasic}~\ref{BasicBasiciv} and Proposition~\ref{propBasicProperties}~\ref{BasicPropertiesiii} that this polynomial is admissible. We shall refer to $\hat f$ as the result of the {\it{Richelot construction associated with the decomposition}} $f = pqr$. It appears that there is a remarkable relation between curves $\mathcal X_f$ and $X_{\hat{f}}$, which we describe in the following proposition.
	
	\begin{proposition}\label{propRichProperties}
		Let $f$ and $\hat{f}$ be as above. Then the following statements hold.
		\begin{enumerate}[label=(\roman*)]
			\item\label{Richi} $2\Per_f \subset \Per_{\hat f} \subset \Per_{f}$ and for all $w_1,w_2 \in \Per_{\hat f}$ the equality $\langle w_1, w_2 \rangle_{f} = 2\langle w_1, w_2\rangle_{\hat f}$ holds.
			\item\label{Richii} Assume that $a_1,a_2,b_1,b_2 \in \Per_f$ is a symplectic basis such that $a_1,a_2\in \Per_{\hat f}$. Then $a_1,a_2,2b_1,2b_2$ is a symplectic basis in $\Per_{\hat f}$. Moreover, such basis always exists.
			\item\label{Richiii} Let $P_1, P_2, Q_1, Q_2, R_1, R_2$ be Weierstrass points of the curve $\mathcal X_f$ listed in such order that $x$-coordinates of $P_1$ and $P_2$ are the roots of $p$, etc. Then the kernel of the isogeny $\Jac_f \to \Jac_{\hat f}$ induced by the multiplication by $2$ operator on $\CC^2$ consists of the elements \[0,\;\; \mathscr A_f((P_1) + (P_2)),\;\; \mathscr A_f((Q_1) + (Q_2)),\;\;\mathscr A_f((R_1) + (R_2)).\]
			\item\label{Richiv} Let $\hat P_1, \hat P_2, \hat Q_1, \hat Q_2, \hat R_1, \hat R_2$ be Weierstrass points of the curve $\mathcal X_{\hat f}$ listed in such order that $x$-coordinates of $\hat P_1$ and $\hat P_2$ are the roots of $[q,r]$, etc. Then the kernel of the isogeny $\Jac_{\hat f} \to \Jac_{f}$ induced by the identity operator on $\CC^2$ consists of the elements \[0,\;\; \mathscr A_{\hat f}((\hat P_1) + (\hat P_2)),\;\; \mathscr A_{\hat f}((\hat P_3) + (\hat P_4)),\;\;\mathscr A_{\hat f}((\hat P_5) + (\hat P_6)).\]
		\end{enumerate}
	\end{proposition}
	\begin{proof}
		It is known that there exists a correspondence $C$ between the curves $\mathcal X_f$ and $\mathcal X_{\hat f}$ that satisfies the following properties (for general theory of correspondences and homomorphisms between Jacobian varieties we refer to~\cite[Section~11.5]{LangeBirkenhake} and~\cite[Section~2.5]{Griffiths}).
		
		\begin{enumerate}[label=(\alph*)]
			\item\label{corra} The pullback of holomorphic differential forms from $\mathcal X_f$ to $\mathcal X_{\hat f}$ with respect to $C$ maps $\omega_j^f$ to $\omega_j^{\hat f}$ for $j = 1,2$.
			\item\label{corrb} $C$ induces an isogeny $R:\Jac_{\hat f} \to \Jac_f$ such that the pullback of the principal polarization on $\Jac_f$ with respect to $R$ equals twice the canonical polarization of $\Jac_{\hat f}$.
			\item\label{corrc} The kernel of isogeny $R$ consists of four points described in~\ref{Richiv}.
			\item\label{corrd} The kernel of the dual isogeny $\hat R:\Jac_f \to \Jac_{\hat f}$ consists of four points described in~\ref{Richiii} (we identify Jacobi varieties with their duals via their canonical principal polarization).
		\end{enumerate}
		
		The listed statements were essentially known to Humbert~\cite{Humbert}, but we would like also to give references to modern accounts. The definition of the correspondence $C$ and statement~\ref{corra} can be found in~\cite[Section~3.1]{BostMestre}. In~\cite[Sections~2.6.2-3]{Chow} an accurate verification of~\ref{corra} and~\ref{corrb} is presented. Finally, statements~\ref{corrc} and~\ref{corrd} are contained in~\cite[Corollary~8.4.14]{Smith}.
		
		From~\ref{corra} it is clear that $R$ is induced by the identity mapping of $\CC^2$, i.e. the diagram
		\[
		\begin{tikzcd}
			\CC^2\arrow[d] \arrow[r, "\mathrm{id}"] & \CC^2\arrow[d] \\
			\Jac_{\hat f}\arrow[r, "R"] & \Jac_f
		\end{tikzcd}
		\]
		commutes. The inclusion $\Per_{\hat f} \subset \Per_f$ follows. The identity $\langle w_1, w_2 \rangle_{f} = 2\langle w_1, w_2\rangle_{\hat f}$ for all $w_1,w_2 \in \Per_{\hat f}$ is now an immediate consequence of~\ref{corrb}. Finally, the kernel of isogeny $R$ can contain only the points of order $2$ by~\ref{corrc}, so $2\Per_f \subset \Per_{\hat f}$. Thus,~\ref{Richi} is proved. To prove statement~\ref{Richii} consider a symplectic basis $a_1,a_2,b_1,b_2 \in \Per_f$ such that $a_1,a_2 \in \Per_{\hat f}$. Then using~\ref{Richi} we can simply calculate the intersection pairings $\langle \cdot, \cdot\rangle_{\hat f}$ of vectors $a_1$, $a_2$, $2b_1$, and $2b_2$ (which belong to $\Per_{\hat f}$ in view of the inclusion $2\Per_f \subset \Per_{\hat f}$) and obtain that these vectors constitute a symplectic basis in $\Per_{\hat f}$. In order to see that such bases exist note that there exists a symplectic basis $a_1,a_2,b_1,b_2$ such that the quadruple of points in $\Jac_f$ listed in~\ref{Richiii} is modulo permutation represented by vectors $0, a_1/2, a_2/2,$ and $(a_1 + a_2)/2$ (it suffices to consider a basis in the group $H_1(\mathcal X_f, \ZZ)$ described in~\cite[\S~IIIa.5]{mumfordII}). By~\ref{corra} and~\ref{corrb} it is easy to see that the dual isogeny $\hat R$ is induced by multiplication by $2$ and by~\ref{corrd} points listed in~\ref{Richiii} belong to $\ker \hat R$. Therefore, the vectors $a_1$ and $a_2$ represent the point $0 \in \Jac_{\hat f}$, that is, $a_1,a_2 \in \Per_{\hat f}$.
		
		The statements~\ref{Richiii} and~\ref{Richiv} are given in~\ref{corrd} and~\ref{corrc} respectively.
	\end{proof}

	\begin{definition}
		Let $f = pqr$ be an admissible polynomial, where $p,q,r \in \mathfrak P_2$ and $\Delta(p,q,r) \ne 0$. Let $\hat{f}$ be defined by~\eqref{eqHatFDef}. The isogeny $\mathscr R_{p,q,r}:\Jac_{\hat f} \to \Jac_f$ induced by the identity map of $\CC^2$ is called the Richelot isogeny.
	\end{definition}
	
	\subsection{Kummer surfaces} 
	
	Recall that the {\it{Kummer variety}} of an Abelian variety $A$ is defined to be the quotient of $A$ by the mapping $z \mapsto -z$ (see~\cite[Section~4.8]{LangeBirkenhake}). We denote the Kummer variety of $A$ by $\mathcal K(A)$. If $A = \Jac_f$, where $f$ is an admissible polynomial, then Kleinian functions of weight $2$ can be used for the embedding of $\mathcal K(\Jac_f)$ into $\CC\PP(3)$.
	
	Let $f$ be an admissible polynomial. We denote by $\mathscr S_f$ the holomorphic mapping $\mathscr S_f:\CC^2 \to \CC^4$ defined by the formula
	\begin{equation*}
		\mathscr S_f(z) = \begin{pmatrix}
			S^f(z) & S_{22}^f(z) & S_{12}^f(z) &  S_{11}^f(z)
		\end{pmatrix}^T
	\end{equation*}
	Note that $\mathscr S_f(z) \ne 0$ for all $z \in \CC^2$ due to \cite[Proposition~3.1~(iii)]{KleinianWeight2}. If $z - z' \in \Per_f$, then there exists a non-zero scalar $c \in \CC$ such that $\mathscr S_f(z') = c \mathscr S_f(z)$. Therefore, there exists a unique holomorphic map $\mathscr K_f:\Jac_f \to \CC \PP(3)$ such that the diagram
	\[
	\begin{tikzcd}
		\Jac_f\arrow[r, "\mathscr K_f"] & \CC \PP(3) \\
		\CC^2\arrow[u]\arrow[r, "\mathscr S_f"] & \CC^4\arrow[u]
	\end{tikzcd}
	\]
	commutes (the vertical arrows are canonical quotient mappings).  It is easy to see that the map $\mathscr K_f$ induces an embedding of the Kummer variety $\mathcal K(\Jac_f)$ in $\CC \PP(3)$. More precisely, from~\cite[Proposition~3.1~(i) and Corollary~3.5]{KleinianWeight2} it follows that $\mathscr K_f$ coincides with the classical embedding of the Kummer variety via theta functions (see~\cite[Section~3.2]{mumfordII}) modulo a linear variable change in $\CC \PP(3)$. The map $\mathscr K_f$ has an advantage over the map defined by theta functions as its composition with $\mathscr A_f$ can be explicitly expressed via rational functions on $\mathcal X_f^{(2)}$ (namely, by formulas~\eqref{eqXiDef} and~\eqref{eqFdef} that define functions $\xi_{jk}^{f}$). In particular, the study of the map $\mathscr K_f \circ \mathscr A_f:\mathcal X_f^{(2)} \to \CC\PP(3)$ can be carried out for curves defined over general fields. The thorough exposition of this theory in full generality can be found in~\cite{CasselsFlynn}.
	
	Now assume that $f = pqr$, where $p,q,r \in \mathfrak P_2$ and $\Delta(p,q,r) \ne 0$. Let $\hat f$ be defined by~\eqref{eqHatFDef}. The Richelot isogeny $\Jac_{\hat f} \to \Jac_f$ induces the holomorphic map between Kummer varieties $\mathcal K(\Jac_{\hat f}) \to \mathcal K(\Jac_f)$. We aim to find an explicit expression for this map using embeddings into $\CC\PP(3)$ defined above. That is, we derive a formula for a map $\mathfrak R_{p,q,r}$ from $\CC\PP(3)$ to itself that makes the diagram
	\begin{equation}\label{eqMainPropertyR}
		\begin{tikzcd}
			\CC\PP(3) \arrow[r, "\mathfrak R_{p,q,r}"] & \CC\PP(3) \\
			\Jac_{\hat f} \arrow[r, "\mathscr R_{p,q,r}"]\arrow[u, "\mathscr K_{\hat f}"] & \Jac_f \arrow[u, "\mathscr K_f"]
		\end{tikzcd}
	\end{equation}
	commutative.
	
	In order to find the expression for $\mathfrak R$ we need to recall several properties of the Kummer varieties $\mathcal K(\Jac_f)$. For convenience we shall call the image of the map $\mathscr K_f:\Jac_f \to \CC\PP(3)$ the {\it{Kummer surface associated with an admissible polynomial}} $f$. To denote it we use the symbol $\mathcal K_f$.
	
	The Kummer surface $\mathcal K_f$ has $16$ {\it{nodes}} -- the images of the points of order $2$ in $\Jac_f$. All non-zero points of order $2$ in $\Jac_f$ are in one-to-one correspondence with unordered pairs of Weierstrass points in $\mathcal X_f$ via the map $\mathscr A_f$. If $P,Q \in \mathcal X_f$ are distinct Weierstrass points, then the corresponding node can be explicitly calculated. Consider a decomposition $f = pg$, such that $x$-coordinates of $P$ and $Q$ are the roots of $p \in \mathfrak P_2$, and $\deg g \le 4$. Then the corresponding node $N_{P,Q} \in \mathcal K_f$ in homogeneous coordinates can be calculated (see~\cite[Section~3.1]{CasselsFlynn}) by the formula
	\begin{equation}\label{eqNodeCoordinates}
		N_{P,Q} = \left(p_2: -p_1:-p_0:-\frac{1}{4}(p_2^2g_0 + p_0p_2g_2 + p_0^2g_4)\right).
	\end{equation}
	This formula is derived by a direct calculation of the functions $\xi^f_{jk}$ at $D = (P) + (Q)$ (if one of the Weierstrass points $P, Q$ is at infinity, then the calculation also involves passing to the limit). The remaining node $N_0$ (that corresponds to $0 \in \Jac_f$) has homogeneous coordinates $(0:0:0:1)$. Moreover, for each node $N = \mathscr K_f(w)$ there is a corresponding symmetry of the Kummer surface $\mathcal K_f$ induced by the map $z \mapsto z + w$ on $\Jac_f$. More precisely, there exists a unique (modulo multiplication by a scalar) matrix $\mathfrak X_N \in \CC^{4 \times 4}$ such that the $\mathscr K_f(z + w) = \mathfrak X_N\mathscr K_f(z)$ for all $z \in \Jac_f$. Clearly, $\mathfrak X_{N_0} = I$. Later we calculate the matrix $\mathfrak X_{N}$ in a specific basis, but it is also possible to give a direct formula (see~\cite[Section~3.2]{CasselsFlynn}).
	
	We also shall use the {\it{tropes}} of Kummer surfaces. A trope is a hyperplane in $\CC\PP(3)$ that contains $6$ distinct nodes of a Kummer surface. Indeed, for each Weierstrass point $P \in \mathcal X_f$ there is a trope $\mathfrak T_{P}$ defined by the equation
	\begin{equation}\label{eqTrope1}
		\mathfrak T_P = \{a = (a_1:a_2:a_3:a_4) \in \CC\PP(3): a_1\beta^2 + a_2 \alpha\beta - a_3\alpha^2 = 0\},
	\end{equation}
	where the $x$-coordinate of $P$ is the root of the polynomial $\alpha x + \beta$. It is straightforward to verify that $\mathfrak T_P$ contains $N_0$ and all the nodes of the form $N_{P,Q}$, where $Q \ne P$ is a Weierstrass point. There are also tropes associated with unordered triples of Weierstrass points of $\mathcal X_f$. That is, if $P,Q,R \in \mathcal X_f$ are distinct Weierstrass points, then there is a decomposition $f = gh$, where $\deg g, \deg h \le 3$ and $x$-coordinates of $P,Q,R$ are exactly the roots of $g$. In this case the trope $\mathfrak T_{P,Q,R}$ is defined (see~\cite[Section~3.7]{CasselsFlynn}) as
	\begin{equation}\label{eqTrope3}
		\mathfrak T_{P,Q,R} = \{a \in \CC\PP(3): (g_2h_0 + g_0h_2)a_1 + (g_3 h_0 + g_0 h_3)a_2 - (g_3 h_1 + g_1 h_3)a_3 + 4a_4 = 0\}.
	\end{equation}
	It can be verified that $\mathfrak T_{P,Q,R}$ contains $N_{U,V}$ if and only if either $U,V \in \{P,Q,R\}$, or $U,V \notin \{P,Q,R\}$. That is, $\mathfrak T_{P,Q,R}$ also contains exactly six nodes of the Kummer surface. Note that a triple $(P,Q,R)$ of distinct Weierstrass points of $\mathcal X_f$ induces the same trope $\mathfrak T_{P,Q,R}$ as the triple of remaining Weierstrass points. Therefore, there are $6$ tropes of the form $\mathfrak T_{P}$ and $10$ tropes of the form $\mathfrak T_{P,Q,R}$.

	\section{Richelot isogeny of Kummer surfaces}\label{secRichelotIsogeny}
	In this section we prove certain properties of Richelot isogenies and Kleinian functions of weight 2 associated with Richelot isogenous curves, that will allow us to find the explicit expression of the Richelot isogeny in Section~\ref{secCalculations}.
	
	\subsection{General form of the Richelot isogeny of Kummer surfaces}\label{subsecRichelotIsogeny}
	
	For brevity we shall fix some notation. Namely, we fix an admissible polynomial $f = pqr$, where $p,q,r \in \mathfrak P_2$ and $\Delta(p,q,r) \ne 0$. Also we let $\hat p = [q,r]$, $\hat q = [r,p]$, $\hat r = [p,q]$, and \[
	\hat f = \frac{1}{4\Delta(p,q,r)}\hat p \hat q \hat r.
	\]
	Also we introduce the nodes $N_p,N_q,N_r,N_0$ of the Kummer surface associated with $f$. Namely, $N_p$ is the node $N_{P,Q}$, where $x$-coordinates of $P$ and $Q$ are the roots of $p$. The nodes $N_q$ and $N_r$ are defined similarly to $N_p$, while $N_0$ is the node that corresponds to $0 \in \Jac_f$. Finally, we introduce similar nodes for the polynomial $\hat f$ -- the nodes $\hat N_p$, $\hat N_q$, $\hat N_r$, and $\hat N_0$. For our considerations it will be important that there are exactly four tropes of the Kummer surface $\mathcal K_f$ that do not contain any of the nodes $N_p,N_q,N_r,N_0$. Indeed, these are the tropes $\mathfrak T_{P,Q,R}$, where $x$-coordinates of the points $P$, $Q$, and $R$ are the roots of $p$, $q$, and $r$ respectively. There are exactly $8$ such triples, but since the complement of such triple induces the same trope, there are exactly $4$ tropes of this kind.
	
	\begin{lemma}\label{lemMainKummerLemma}
		Let $\mathrm{Sq}:\CC\PP(3) \to \CC\PP(3)$ denote the mapping induced by squaring all the homogeneous coordinates, i.e.
		\[
		\mathrm{Sq}(a) = (a_1^2:a_2^2:a_3^2:a_4^2).
		\]
		Then there exist matrices $\mathfrak B, \hat{\mathfrak  B} \in \CC^{4 \times 4}$ satisfying the following properties.
		\begin{enumerate}[label=(\roman*)]
			\item\label{mainKummeri} For all $z \in \Jac_{\hat f}$ the equality
			\[
			\mathfrak B\mathscr K_f(\mathscr R_{p,q,r}(z)) = \mathrm{Sq}\left(\hat {\mathfrak B}\mathscr K_{\hat f}(z)\right).
			\]
			\item\label{mainKummerii} Matrix $\hat {\mathfrak B}$ is invertible and the hyperplanes in $\CC\PP(3)$ induced by its rows are invariant with respect to the symmetries $\mathfrak X_{\hat N_p}$, $\mathfrak X_{\hat N_q}$, $\mathfrak X_{\hat N_r}$.
			\item\label{mainKummeriii} Matrix $\mathfrak B$ is invertible and the hyperplanes in $\CC\PP(3)$ induced by its rows are the tropes of the Kummer surface $\mathcal K_f$ that do not contain any of the nodes $N_0, N_p, N_q, N_r$.
		\end{enumerate}
	\end{lemma}
	\begin{proof}
		Let $a_1,a_2,b_1,b_2 \in \Per_f$ be a symplectic basis such that $a_1,a_2,2b_1,2b_2$ is a symplectic basis in $\Per_{\hat f}$. Let $A$ and $B$ be the $2\times2$ matrices with columns $a_1,a_2$ and  $b_1,b_2$ respectively. From Proposition~\ref{propRichProperties}~\ref{Richiii} it follows that the nodes $N_p,N_q,N_r$ correspond modulo permutation to the points $a_1/2, a_2/2$, and $(a_1 + a_2)/2$ of order $2$ in $\Jac_f$. Finally, let $\Omega = A^{-1}B$. Similarly, from Proposition~\ref{propRichProperties}~\ref{Richiv} we conclude that the nodes $\hat N_p,\hat N_q,\hat N_r$ correspond modulo permutation to the points $b_1, b_2,$ and $b_1 + b_2$ of order $2$ in $\Jac_{\hat f}$. Finally, let $T:R_2^{\Omega} \to \mathfrak S_f$ and $\hat T:R_2^{2\Omega} \to \mathfrak S_{\hat f}$ be the isomorphisms defined by~\eqref{eqTisomorphism} for the polynomials $f$ and $\hat f$ respectively and the foregoing choice of symplectic bases.
		
		Now we choose convenient bases in spaces $R_2^\Omega$ and $R_2^{2\Omega}$. Namely, let
		\[
		2\alpha_1 = \begin{pmatrix}
			0 \\ 0
		\end{pmatrix},\;\;
		2\alpha_2 = \begin{pmatrix}
			1 \\ 0
		\end{pmatrix},\;\;
		2\alpha_3 = \begin{pmatrix}
			0 \\ 1
		\end{pmatrix},\;\;
		2\alpha_4 = \begin{pmatrix}
			1 \\ 1
		\end{pmatrix}
		\]
		and consider functions
		\[
		\hat \phi_j(z) = \theta\begin{bmatrix}
			0 \\ \alpha_j
		\end{bmatrix}(z, \Omega),\;\;\phi_j = \left(\theta\begin{bmatrix}
			0 \\ \alpha_j
		\end{bmatrix}(z, \Omega)\right)^2.
		\]
		for $j = 1,2,3,4$. It is clear that $\phi_j \in R_2^{\Omega}$ and $\hat \phi_j \in R_2^{2\Omega}$ for all $j = 1,2,3,4$. Moreover, functions $\hat \phi_j$ constitute a basis in $R_2^{2\Omega}$ by~\cite[Proposition~II.1.3]{mumfordI}. To see that the functions $\phi_j$ form a basis in $R_2^\Omega$ we apply a version of the addition formula (see, e.g.~\cite[\S~IV.1]{Igusa}) to obtain
		\begin{multline}\label{eqTheta1Basis}
			\phi_j(z) = \theta(z - \alpha_j, \Omega)\theta(z + \alpha_j, \Omega) = \sum_{k = 1}^4 \theta\begin{bmatrix}
				\alpha_k \\ 0
			\end{bmatrix}(2z, 2\Omega)\theta\begin{bmatrix}
				\alpha_k \\ 0 
			\end{bmatrix}(2\alpha_j, \Omega) = \\ \sum_{k = 1}^4 \exp(2\pi i \alpha_k^T\alpha_j)\theta\begin{bmatrix}
				\alpha_k \\ 0 
			\end{bmatrix}(0, \Omega)\theta\begin{bmatrix}
				\alpha_k \\ 0
			\end{bmatrix}(2z, 2\Omega)
		\end{multline}
		From~\cite[Proposition~II.1.3]{mumfordI} we know that the functions
		\[
		\psi_j(z) = \theta\begin{bmatrix}
			\alpha_j \\ 0
		\end{bmatrix}(2z, 2\Omega)
		\]
		for $j = 1,2,3,4$ form a basis in $R_2^\Omega$. The equality~\eqref{eqTheta1Basis} can be rewritten in the form
		\[
		\begin{pmatrix}
			\phi_1(z)\\ \phi_2(z)\\ \phi_3(z)\\ \phi_4(z)
		\end{pmatrix} = \begin{pmatrix}
			1 & 1 & 1 & 1\\
			1 & -1 & 1 & -1 \\
			1 & 1 & -1 & -1 \\
			1 & -1 & -1 & 1
		\end{pmatrix}\begin{pmatrix}
			\theta_1 & 0 & 0 & 0\\
			0 & \theta_2 & 0 & 0\\
			0 & 0 & \theta_3 & 0 \\
			0 & 0 & 0 & \theta_4
		\end{pmatrix}\begin{pmatrix}
			\psi_1(z) \\ \psi_2(z) \\ \psi_3(z) \\ \psi_4(z)
		\end{pmatrix}.
		\]
		where we let $\theta_j$ for $j = 1,2,3,4$ denote
		\[
		\theta_j = \theta\begin{bmatrix}
			\alpha_j \\ 0 
		\end{bmatrix}(0, 2\Omega).
		\]
		Clearly, for $\phi_1,\phi_2,\phi_3,\phi_4$ to be a basis in $R_2^\Omega$ it suffices to prove that $\theta_j \ne 0$ for all $j$. But this is clear from the Thomae formula~\cite[Theorem~III.8.1]{mumfordII}, which implies that for a period matrix of a genus $2$ Riemann surface all even theta constants do not vanish.
		
		With this preparation we finally let $\mathfrak B$ and $\hat {\mathfrak B}$ denote the $4 \times 4$ matrices such that the equalities
		\[
		\begin{pmatrix}
			T\phi_1(z) \\ T\phi_2(z) \\T\phi_3(z) \\T\phi_4(z)
		\end{pmatrix} = \mathfrak B\mathscr S_f(z),\;\;\begin{pmatrix}
			\hat T\hat \phi_1(z) \\ \hat T\hat \phi_2(z) \\ \hat T\hat \phi_3(z) \\ \hat T\hat \phi_4(z)
		\end{pmatrix} = \hat {\mathfrak B}\mathscr S_{\hat f}(z)
		\]
		hold for all $z \in \CC^2$. Since $\phi_j(z) = \hat \phi_j(z)^2$ it is easy to see that $Tf_j(z) = c(z)T\hat \phi_j(z)^2$ for $j =1,2,3,4$ with a suitable scalar $c(z)$. Thus, the statement~\ref{mainKummeri} holds. From the definition of the functions $\hat \phi_j$ it follows that the shift by one of the vectors $b_1, b_2, b_1 + b_2$ multiplies the vector $\begin{pmatrix} \hat T\hat f_1(z) & \hat T\hat f_2(z) & \hat T\hat f_3(z) &\hat T\hat f_4(z) \end{pmatrix}^T$ by a common scalar factor and an independent of $z$ diagonal matrix comprised of values $\pm 1$ on the diagonal. Thus, in this basis the operators $\mathfrak X_{\hat N_p}$, $\mathfrak X_{\hat N_q}$, and $\mathfrak X_{\hat N_r}$ are diagonal. Thus,~\ref{mainKummerii} is proved. It remains to prove~\ref{mainKummeriii}. Note that the intersection of the hyperplane defined by $j$-th row of $B$ with the Kummer surface of $f$ coincides with the set $\{\mathscr K_f(z): z \in \CC^2, T\phi_j(z) = 0\}$. It is known that a theta function associated with a Riemann surface of genus $2$ with half-integer characteristic $\begin{bmatrix} \alpha & \beta \end{bmatrix}^T$ vanishes exactly on six points of order $2$ (namely, on the set of odd theta characteristics shifted by the vector $\Omega \alpha + \beta$). Thus, the hyperplane defined by $j$-th row of $B$ is a trope. Moreover, $\phi_j$ does not vanish at the points $\alpha_1, \alpha_2, \alpha_3, \alpha_4$, and by definition of $T$ it is clear that $T\phi_j$ does not vanish at $0, a_1/2, a_2/2,$ and $(a_1 + a_2)/2$. Thus, the trope defined by any row of $B$ does not contain any of the nodes $N_p, N_q, N_r,$ and $N_0$. Thus, we have proved~\ref{mainKummeriii}.
	\end{proof}
	\begin{remark}
		It is easy to see that the conditions~\ref{mainKummerii} and~\ref{mainKummeriii} of Lemma~\ref{lemMainKummerLemma} determine the matrices $\mathfrak B$ and $\hat {\mathfrak B}$ up to a permutation of rows and multiplication of each row by a non-zero scalar. Moreover,
		the property~\ref{mainKummeri} of Lemma~\ref{lemMainKummerLemma} stays true if we permute the rows of $\hat {\mathfrak B}$ and multiply them by scalars and perform the corresponding transformation of the matrix $\mathfrak B$ (with the squared scalars). Thus, we can formulate a more precise statement than Lemma~\ref{lemMainKummerLemma}: for any matrix $\hat {\mathfrak B}$ that satisfies~\ref{mainKummerii} there exists a matrix $\mathfrak B$ such that statements~\ref{mainKummeri} and~\ref{mainKummeriii} hold.
	\end{remark}

	\subsection{Special properties of the components of \texorpdfstring{$\mathfrak R_{p,q,r}$}{mathfrak R (p,q,r)}}
	
	From Lemma~\ref{lemMainKummerLemma} it follows that $\mathfrak R_{p,q,r}$ from diagram~\eqref{eqMainPropertyR} is indeed quadratic. To derive further information it will be useful to introduce the components of $\mathfrak R$.
	
	\begin{definition}\label{defRichelotRepresentation}
		Let $f = pqr$ be an admissible polynomial, where $p,q,r \in \mathfrak P_2$ and $\Delta(p,q,r) \ne 0$. Let $\hat f$ be the result of the Richelot construction associated with the decomposition $f = pqr$. We say that a quadruple of symmetric $4 \times 4$ matrices $(\mathfrak A_1, \mathfrak A_2, \mathfrak A_3,\mathfrak A_4)$ represents the Richelot isogeny $\mathscr R_{p,q,r}$ if the homogeneous map of order $2$ from $\CC^4$ to itself given by the formula
		\[
		a \mapsto \begin{pmatrix}
			a^T\mathfrak A_1a \\ a^T\mathfrak  A_2a \\ a^T\mathfrak A_3a \\ a^T \mathfrak A_4a
		\end{pmatrix}
		\]
		maps non-zero vectors to non-zero vectors and if the induced by it mapping $\mathfrak A:\CC\PP(3) \to \CC\PP(3)$ makes the
		diagram
		\[
		\begin{tikzcd}
			\CC\PP(3) \arrow[r, "\mathfrak A"] & \CC\PP(3) \\
			\Jac_{\hat f} \arrow[r, "\mathscr R_{p,q,r}"]\arrow[u, "\mathscr K_{\hat f}"] & \Jac_f \arrow[u, "\mathscr K_f"]
		\end{tikzcd}
		\]
		commutative.
	\end{definition}
	
	Now we derive some properties satisfied by the matrices $\mathfrak A_j$, $j = 1,2,3,4$ from Definition~\ref{defRichelotRepresentation}. To derive them we find a relation between vector-valued functions $\mathscr S_f$ and $\mathscr S_{\hat f}$ using these matrices and apply Taylor expansions~\eqref{eqTaylorExpansions}. 
	
	\begin{lemma}\label{lemEtaTransformation}
		Let $f = pqr$ be an admissible polynomial and assume that $p,q,r \in \mathfrak P_2$ and $\Delta(p,q,r) \ne 0$. Then there exists a uniquely defined matrix $\mathfrak H^{p,q,r} \in \CC^{2 \times 2}$ such that 
		\begin{equation}\label{eqEtaTransformation}
			\eta^f(w) = 2\eta^{\hat f}(w) + \mathfrak H^{p,q,r}w
		\end{equation}
		for all $w \in \Per_{\hat f}$. Moreover, $\mathfrak H^{p,q,r}$ is symmetric.
	\end{lemma}
	\begin{proof}
		Let $a_1,a_2,b_1,b_2 \in \Per_f$ be a symplectic basis such that $a_1,a_2,2b_1,2b_2$ is a symplectic basis in $\Per_{\hat f}$ (see Proposition~\ref{propRichProperties}~\ref{Richii}). Let $A$ and $B$ be the matrices with columns $a_1,a_2$ and $b_1,b_2$ respectively. Further let $\eta_A$ and $\eta_B$ have the columns $\eta^f(a_1), \eta^f(a_2)$ and $\eta^f(b_1), \eta^f(b_2)$ respectively. Finally, let $\hat\eta_A$ and $\hat\eta_B$ have the columns $\eta^{\hat f}(a_1), \eta^{\hat f}(a_2)$ and $\eta^{\hat f}(2b_1), \eta^{\hat f}(2b_2)$ respectively. With these definitions it is clear that if $\mathfrak H^{p,q,r}$ exists, then it satisfies
		\[
		\eta_A = 2\hat \eta_A + \mathfrak H^{p,q,r}A,
		\]
		implying that $\mathfrak H^{p,q,r} = (\eta_A - 2\hat\eta_A)A^{-1}$. This verifies uniqueness of $\mathfrak H^{p,q,r}$. Now we prove that the matrix $(\eta_A - 2\hat\eta_A)A^{-1}$ satisfies~\eqref{eqEtaTransformation} for all $w \in \Per_{\hat f}$. Clearly it suffices to verify this for $w \in \{a_1,a_2,b_1,b_2\}$ and by definition it is clear for $w \in \{a_1,a_2\}$. It remains to verify the equality
		\[
		2\eta_B = 2\hat \eta_B + 2(\eta_A - 2\hat\eta_A)A^{-1}B.
		\]
		Using the fact that $A^{-1}B$ is symmetric and applying Legendre's identities~\eqref{eqLegendre} we get
		\begin{multline*}
			(\eta_B - \hat \eta_B) - (\eta_A - 2\hat \eta_A)A^{-1}B = (\eta_B - \hat \eta_B) - (\eta_A - 2\hat \eta_A)B^T\left(A^T\right)^{-1} = \\ (\eta_B - \hat \eta_B) - (2i\pi I + \eta_BA^T - 2i\pi I - \hat\eta_B A^T)\left(A^T\right)^{-1} = (\eta_B - \hat \eta_B) - (\eta_B - \hat \eta_B) = 0.
		\end{multline*}
		Thus,~\eqref{eqEtaTransformation} holds for $w \in \{b_1,b_2\}$ and, consequently, for all $w \in \Per_{\hat f}$. It remains to show that $(\eta_A - 2\hat\eta_A)A^{-1}$ is symmetric. From Legendre's identities~\eqref{eqLegendre} it follows that
		\[
		\eta_A^TA = A^T\eta_A,
		\]
		which implies \[
		\left(A^T\right)^{-1}\eta_A^T = \eta_A A^{-1},
		\]
		so $(\eta_A A^{-1})^T = \eta_AA^{-1}$. By the same reasoning the matrix $\hat \eta_AA^{-1}$ is also symmetric. Thus, $(\eta_A - 2\hat\eta_A)A^{-1}$ is symmetric.
	\end{proof}
	\begin{lemma}\label{lemMainTransformation}
		Let $f = pqr$ be an admissible polynomial and assume that $p,q,r \in \mathfrak P_2$ and $\Delta(p,q,r) \ne 0$. Assume that a quadruple $(\mathfrak A_1,\mathfrak A_2, \mathfrak A_3,\mathfrak A_4)$ represents the Richelot isogeny $\mathscr R_{p,q,r}$. Then the following statements hold.
		
		\begin{enumerate}[label=(\roman*)]
			\item\label{MainTransfi} There exists a complex number $c \ne 0$ such that the matrices $\mathfrak A_j$ have the form \begin{equation}\label{eqMatrixLastRows}
				\begin{gathered}
					\mathfrak A_1 =  \begin{pmatrix}
						* & * & * & c\\ * & * & * & 0\\  * & * & * & 0\\ c & 0 & 0 & 0
					\end{pmatrix},\;\;\mathfrak A_2 = \begin{pmatrix}
						* & * & * & 0\\ * & * & * & c\\  * & * & * & 0\\ 0 & c & 0 & 0
					\end{pmatrix},\;\;
					\mathfrak A_3 = \begin{pmatrix}
						* & * & * & 0\\ * & * & * & 0\\  * & * & * & c\\ 0 & 0 & c & 0
					\end{pmatrix},\\\mathfrak A_4 = \begin{pmatrix}
						* & * & * & -c\mathfrak H^{p,q,r}_{11}\\ * & * & * & -c\mathfrak H^{p,q,r}_{12}\\  * & * & * & c\mathfrak H^{p,q,r}_{22}\\ -c\mathfrak H^{p,q,r}_{11} & -c\mathfrak H^{p,q,r}_{12} & c\mathfrak H^{p,q,r}_{22} & 2c
					\end{pmatrix},
				\end{gathered}
			\end{equation}
			where $*$ denotes any complex number and
			\[
			\mathfrak H^{p,q,r} = \begin{pmatrix}
				\mathfrak H_{11}^{p,q,r} & \mathfrak H_{12}^{p,q,r} \\ 
				\mathfrak H_{12}^{p,q,r} & \mathfrak H_{22}^{p,q,r}
			\end{pmatrix}
			\]
			is defined in Lemma~\ref{lemEtaTransformation}.
			\item\label{MainTransfii} For all $z \in \CC^2$ the equality
			\[
			\mathscr S_f(z) = \frac{\exp\left(z^T \mathfrak H^{p,q,r}z\right)}{2c} \begin{pmatrix}
				\left(\mathscr S_{\hat f}(z)\right)^T \mathfrak A_1\mathscr S_{\hat f}(z) \\
				\left(\mathscr S_{\hat f}(z)\right)^T \mathfrak A_2\mathscr S_{\hat f}(z) \\
				\left(\mathscr S_{\hat f}(z)\right)^T \mathfrak A_3\mathscr S_{\hat f}(z) \\
				\left(\mathscr S_{\hat f}(z)\right)^T \mathfrak A_4\mathscr S_{\hat f}(z)
			\end{pmatrix}
			\]
			holds, where $c$ is the number from~\ref{MainTransfi}.
		\end{enumerate}
	\end{lemma}
	\begin{proof}
		From Definition~\ref{defRichelotRepresentation} it follows that for all $z \in \CC^2$ there exists a number $\phi(z) \ne 0$ such that \[
		\mathscr S_f(z) = \phi(z)\begin{pmatrix}
			\left(\mathscr S_{\hat f}(z)\right)^T \mathfrak A_1\mathscr S_{\hat f}(z) \\
			\left(\mathscr S_{\hat f}(z)\right)^T \mathfrak A_2\mathscr S_{\hat f}(z) \\
			\left(\mathscr S_{\hat f}(z)\right)^T \mathfrak A_3\mathscr S_{\hat f}(z) \\
			\left(\mathscr S_{\hat f}(z)\right)^T \mathfrak A_4\mathscr S_{\hat f}(z)
		\end{pmatrix}.
		\]
		Obviously, $\phi$ is a holomorphic function and by comparing quasiperiodicity properties~\eqref{eqSpaceDef} of functions from the spaces $\mathfrak S_f$ and $\mathfrak S_{\hat f}$ we obtain that
		\[
			\phi(z + w) = \exp\left[\left(2\eta^f(w) - 4\eta^{\hat f}(w)\right)^T\left(z + \frac{w}{2}\right)\right]\phi(z) = \exp\left[2w^T\mathfrak H^{p,q,r}\left(z + \frac{w}{2}\right)\right]\phi(z)
		\]
		for all $w \in \Per_{\hat f}$. It follows that the function \[
		z \mapsto \exp[-z^T\mathfrak H^{p,q,r}z]\phi(z)
		\]
		is $\Per_{\hat f}$-periodic and, therefore, constant. Thus, we proved that for some constant $\alpha \ne 0$ the equality
		\[
		\mathscr S_f(z) = \alpha  \exp\left(z^T\mathfrak H^{p,q,r}z\right)\begin{pmatrix}
			\left(\mathscr S_{\hat f}(z)\right)^T \mathfrak A_1\mathscr S_{\hat f}(z) \\
			\left(\mathscr S_{\hat f}(z)\right)^T \mathfrak A_2\mathscr S_{\hat f}(z) \\
			\left(\mathscr S_{\hat f}(z)\right)^T \mathfrak A_3\mathscr S_{\hat f}(z) \\
			\left(\mathscr S_{\hat f}(z)\right)^T \mathfrak A_4\mathscr S_{\hat f}(z)
		\end{pmatrix}
		\]
		holds. Finally, by substituting the expansions~\eqref{eqTaylorExpansions} into this equality and calculating all terms of order $\le 2$ we get~\eqref{eqMatrixLastRows} with $c = 1/2\alpha$.
	\end{proof}
	
	Using the restrictions on matrices $\mathfrak A_j$ from Lemma~\ref{lemMainTransformation}~\ref{MainTransfi} we find an important relation between matrices $\mathfrak B$ and $\hat{\mathfrak B}$ from Lemma~\ref{lemMainKummerLemma}. This relation will allow us to overcome the indeterminacy of the description of these matrices given in Lemma~\ref{lemMainKummerLemma}, as the conditions~\ref{mainKummerii} and~\ref{MainTransfii} determine $\mathfrak B$ and $\hat{\mathfrak B}$ only up to the permutation of rows and multiplication of the rows by arbitrary non-zero scalars (see remark that follows Lemma~\ref{lemMainKummerLemma}).
	
	\begin{lemma}\label{lemFirstThreeRows}
		Assume that $\mathfrak B$ and $\hat {\mathfrak B}$ are invertible $4\times 4$ matrices that satisfy the condition~\ref{mainKummeri} of Lemma~\ref{lemMainKummerLemma}. Moreover, let $\mathfrak b = \begin{pmatrix}
			\mathfrak b_1 & \mathfrak b_2 &\mathfrak b_3 &\mathfrak b_4
		\end{pmatrix}^T$ denote the last column of the matrix $\hat {\mathfrak B}$. Then all components of $\mathfrak b$ do not equal to zero and there exists a non-zero scalar $c$ such that the matrix
		\[
		c\mathfrak B^{-1} - \hat {\mathfrak B}^{-1}\mathrm{diag}( {\mathfrak b})^{-1}
		\]
		has non-zero entries only in the last row. 
	\end{lemma}
	\begin{proof}
		Consider the map $\mathfrak A$ defined by the formula
		\begin{equation}\label{eqMapA}
			\mathfrak A(a) = \mathfrak B^{-1}\mathrm{Sq(\hat {\mathfrak B}a)}.
		\end{equation}
		Clearly, all components of this map are quadratic forms in the entries of $a$, so for suitable symmetric $\mathfrak A_1,\mathfrak A_2,\mathfrak A_3,\mathfrak A_4 \in \CC^{4 \times 4}$ we have the identity
		\[
		\mathfrak A(a) = \begin{pmatrix}
			a^T \mathfrak A_1a \\ a^T \mathfrak A_2a \\ a^T\mathfrak A_3a \\ a^T\mathfrak A_4a
		\end{pmatrix}, \;\;a \in \CC^4.
		\]
		Clearly, property~\ref{mainKummeri} from Lemma~\ref{lemMainKummerLemma} implies that the matrices $\mathfrak A_j$, $j = 1,2,3,4$ represent the Richelot isogeny. Now we calculate the last rows of these matrices directly from~\eqref{eqMapA}. With the notation
		\[
		\mathfrak B^{-1} = \begin{pmatrix}
			\beta_{11} & \beta_{12} & \beta_{13} & \beta_{14} \\
			\beta_{21} & \beta_{22} & \beta_{23} & \beta_{24} \\
			\beta_{31} & \beta_{32} & \beta_{33} & \beta_{34} \\
			\beta_{41} & \beta_{42} & \beta_{43} & \beta_{44}
		\end{pmatrix},\;\;\hat {\mathfrak B} = \begin{pmatrix}
			b_1 \\ b_2 \\ b_3 \\ b_4
		\end{pmatrix},\;\;b_1,b_2,b_3,b_4 \in \CC^{1\times 4}
		\]
		it is easy to find that
		\begin{equation}\label{eqLastRows}
			\mathfrak A_j = \begin{pmatrix}
				* \\ *\\ * \\ \displaystyle\sum_{k = 1}^4\beta_{jk}\mathfrak b_kb_k 
			\end{pmatrix},
		\end{equation}
		where $*$ denotes arbitrary row of $4$ numbers. From Lemma~\ref{lemMainTransformation}~\ref{MainTransfi} it follows, in particular, that the last rows of $A_j$, $j = 1,2,3,4$ are linearly independent. In view of~\eqref{eqLastRows} this is possible only if all components of $\mathfrak b$ are not equal to zero. Moreover, it follows that there exists $c \ne 0$ such that $c\sum_{k = 1}^4\beta_{jk}\mathfrak b_kb_k$ is the $j$-th row of the identity matrix for $j = 1,2,3$. By definition this means that first three rows of the matrix
		\[
		c\mathfrak B^{-1}\mathrm{diag}(\mathfrak b)
		\]
		are precisely the first three rows of the matrix $\hat {\mathfrak B}^{-1}$.
	\end{proof}
	
	\section{Explicit expression for the Richelot isogeny}\label{secCalculations}
	In this section we calculate the Richelot isogeny using Lemmas~\ref{lemMainKummerLemma} and~\ref{lemFirstThreeRows}. The approach to this calculation is very direct, i.e. we calculate the matrices $\mathfrak B$ and $\hat{\mathfrak B}$ from Lemma~\ref{lemMainKummerLemma} using their properties~\ref{mainKummerii} and~\ref{mainKummeriii} modulo permutation of rows and multiplication of them by scalars. The mentioned indeterminacy can be eliminated using Lemma~\ref{lemFirstThreeRows}. Since some of the calculations below are very cumbersome we performed computer verifications of several formulas in Maple\footnote{The code can be found in \url{https://github.com/Matvey1/Genus2KleinianFunctions}.}.

	\subsection{Calculation of the rows of matrices \texorpdfstring{$\mathfrak B$}{mathfrak B} and \texorpdfstring{$\hat{\mathfrak B}$}{hat mathfrak B}}
	
	In this subsection we keep the notation from Subsection~\ref{subsecRichelotIsogeny}, namely, the objects $f$, $\hat f$, $p, q, r,$ $\hat p, \hat q, \hat r$, $N_p, N_q, N_r, N_0$, $\hat N_p, \hat N_q, \hat N_r, \hat N_0$.
	
	At first we calculate the rows of the matrix $\hat{\mathfrak B}$. For this it will be convenient to introduce a specific basis in $\CC^4$. Namely, let
	\begin{equation}\label{eqMatrixC}
		\mathfrak C_{p,q,r} = \begin{pmatrix}
			\hat p_2 & \hat q_2 & \hat r_2 & 0 \\
			-\hat p_1 & -\hat q_1 & -\hat r_1 & 0 \\
			-\hat p_0 & -\hat q_0 & -\hat r_0 & 0 \\
			\phi_{1}(p,q,r) &
			\phi_2(p,q,r) &
			\phi_3(p,q,r) & -\dfrac{1}{16 \Delta(p,q,r)}
		\end{pmatrix},
	\end{equation}
	where
	\[
	\begin{gathered}
		\phi_1(p,q,r) = -\frac{1}{16\Delta(p,q,r)}(\hat p_2^2 \hat q_0 \hat r_0 + \hat p_2 \hat p_0(\hat q_0 \hat r_2 + \hat q_1 \hat r_1 + \hat q_2 \hat r_0) + \hat p_0^2\hat q_2\hat r_2),\\
		\phi_2(p,q,r) = -\frac{1}{16\Delta(p,q,r)}(\hat q_2^2 \hat r_0 \hat p_0 + \hat q_2 \hat q_0(\hat r_0 \hat p_2 + \hat r_1 \hat p_1 + \hat r_2 \hat p_0) + \hat q_0^2\hat r_2\hat p_2),\\
		\phi_3(p,q,r) = -\frac{1}{16\Delta(p,q,r)}(\hat r_2^2 \hat p_0 \hat q_0 + \hat r_2 \hat r_0(\hat p_0 \hat q_2 + \hat p_1 \hat q_1 + \hat p_2 \hat q_0) + \hat r_0^2\hat p_2\hat q_2).
	\end{gathered}
	\]
	That is, columns of $\mathfrak C_{p,q,r}$ are coordinate representations of the nodes $\hat N_p$, $\hat N_q$, $\hat N_r$, $\hat N_0$ by~\eqref{eqNodeCoordinates}. It is easy to see that $\mathfrak C_{p,q,r}$ is invertible as $\Delta(\hat p, \hat q, \hat r) \ne 0$ by Proposition~\ref{propBasicProperties}~\ref{BasicPropertiesii}.
	
	\begin{lemma}\label{lemSymmetriesFormula}
		The matrices $\mathfrak X_{\hat N_p}$, $\mathfrak X_{\hat N_q}$, $\mathfrak X_{\hat N_r}$ can be calculated by the formulas
		\[
		\mathfrak X_{\hat N_p} = \mathfrak C_{p,q,r} \begin{pmatrix}
			0 & 0 & 0 & 1 \\
			0 & 0 & \Res(\hat p, \hat r) & 0 \\
			0 & \Res(\hat p,\hat q) & 0 & 0 \\
			\Res(\hat p, \hat r)\Res(\hat p,\hat q)& 0 & 0 & 0
		\end{pmatrix} \mathfrak C_{p,q,r}^{-1},
		\]
		\[
		\mathfrak X_{\hat N_q} = \mathfrak C_{p,q,r} \begin{pmatrix}
			0 & 0 & \Res(\hat q, \hat r) & 0 \\
			0 & 0 & 0 & 1 \\
			\Res(\hat q, \hat p) & 0 & 0 & 0 \\
			0 & \Res(\hat q, \hat p)\Res(\hat q,\hat r)& 0 & 0
		\end{pmatrix} \mathfrak C_{p,q,r}^{-1},
		\]
		\[
		\mathfrak X_{\hat N_r} = \mathfrak C_{p,q,r} \begin{pmatrix}
			0 & \Res(\hat r, \hat q) & 0 & 0 \\
			\Res(\hat r, \hat p) & 0 & 0 & 0 \\
			0 & 0 & 0 & 1 \\
			0 & 0 & \Res(\hat r, \hat p)\Res(\hat r, \hat q) & 0
		\end{pmatrix} \mathfrak C_{p,q,r}^{-1}
		\]
		modulo multiplication by non-zero scalars.
	\end{lemma}
	\begin{proof}
		We prove only the formula for $\mathfrak X_{N_p}$ as the other matrices are obtained in the same way. Let $\hat w_p$, $\hat w_q$, and $\hat w_r$ denote the points of order $2$ that correspond to the nodes $\hat N_p$, $\hat N_q$ and $\hat N_r$. From the arithmetic of theta characteristics described in~\cite[Proposition~6.1]{mumfordII} it follows that \[\hat w_p + \hat w_p = 0,\;\;\hat w_p + \hat w_q = \hat w_r,\;\;\hat w_p + \hat w_r = \hat w_q.\;\;\] From the definitions of $\mathfrak C_{p,q,r}$ and $\mathfrak X_{\hat N_p}$ it follows that there exist non-zero scalars $\alpha$, $\beta$, $\gamma$, and $\delta$ such that
		\[
		\mathfrak X_{\hat N_p} = \mathfrak C_{p,q,r} \begin{pmatrix}
			0 & 0 & 0 & \alpha \\
			0 & 0 & \beta & 0 \\
			0 & \gamma & 0 & 0 \\
			\delta & 0 & 0 & 0
		\end{pmatrix} \mathfrak C_{p,q,r}^{-1}.
		\]
		From the fact that $\mathfrak X^2_{p,q,r}$ is proportional to the identity matrix we conclude that $\alpha\delta = \beta\gamma$. Other relations between the constants $\alpha$, $\beta$, $\gamma$, and $\delta$ can be obtained by observing the action of $\mathfrak X^2_{p,q,r}$ on the tropes. At first let $\hat P_1$ and $\hat P_2$ denote the Weierstrass points of $\mathcal X_{\hat f}$ that correspond to the roots of $\hat p$. Consider the decomposition $\hat p = \hat p^{(1)}\hat p^{(2)}$, where $\hat p^{(1)}, \hat p^{(2)} \in \mathfrak P_1$ and the $x$-coordinate of $P_j$ is the root of $p^{(j)}$, $j = 1,2$. Then the symmetry $\mathfrak X_{\hat N_p}$ exchanges the tropes $\mathfrak T_{\hat P_1}$ and $\mathfrak T_{\hat P_2}$. By a simple calculation from~\eqref{eqMatrixC} and~\eqref{eqTrope1} it can be verified that 
		\[
		\begin{gathered}
			\mathfrak C_{p,q,r}a \in \mathfrak T_{\hat P_1} \Leftrightarrow \Res\left(\hat q, \hat p^{(1)}\right)a_2 + \Res\left(\hat r, \hat p^{(1)}\right)a_3 = 0, \\
			\mathfrak C_{p,q,r}a \in \mathfrak T_{\hat P_2} \Leftrightarrow \Res\left(\hat q, \hat p^{(2)}\right)a_2 + \Res\left(\hat r, \hat p^{(2)}\right)a_3 = 0.
		\end{gathered}
		\]
		Since these hyperplanes have to be exchanged by $\mathfrak X_{\hat N_p}$, in view of the simple identity $\Res(g, p^{(1)})\Res(g, p^{(2)}) = \Res(g, p)$ that holds for all $g \in \mathfrak P_2$, we obtain the relation
		\[
		\gamma\Res(\hat p,\hat r) - \beta\Res(\hat p, \hat q) = 0.
		\]
		Now consider the Weierstrass points $\hat Q_1$ and $\hat Q_2$ that correspond to the decomposition $\hat q = \hat q^{(1)}\hat q^{(2)}$. Then the matrix $\mathfrak X_{\hat N_p}$ maps the trope $\mathfrak T_{\hat Q_1}$ to the trope $\mathfrak T_{\hat P_1, \hat P_2, \hat Q_1}$. From~\eqref{eqMatrixC},~\eqref{eqTrope1}, and~\eqref{eqTrope3} it can be established that 
		\[
		\mathfrak C_{p,q,r}a \in \mathfrak T_{\hat Q_1} \Leftrightarrow \Res\left(\hat p, \hat q^{(1)}\right)a_1 + \Res\left(\hat r, \hat q^{(1)}\right)a_3 = 0,\]
		\begin{equation}\label{eqTropeAfterCoordChange}
			\mathfrak C_{p,q,r}a \in \mathfrak T_{\hat P_1 \hat P_2 \hat Q_1} \Leftrightarrow \Res\left(\hat p, \hat q^{(2)}\right)\Res\left(r, q^{(1)}\right)a_2 + a_4 = 0.
		\end{equation}
		We note that the formula \eqref{eqTropeAfterCoordChange} was verified using Maple. From this we obtain that \[
		\Res(\hat p, \hat q)\beta - \delta = 0.
		\]
		Finally, since the matrix $\mathfrak X_{\hat N_p}$ is defined modulo multiplication by a constant we can set $\alpha = 1$. The relations that we have proved then imply
		\[
		\beta = \Res(\hat p, \hat r),\;\;\gamma = \Res(\hat p, \hat q),\;\;\delta = \Res(\hat p, \hat r)\Res(\hat p, \hat q).
		\]
		
	\end{proof}
	
	\begin{corollary}\label{corColumnsOfHatB}
		Let $\delta_p, \delta_q, \delta_r \in \CC$ be any numbers such that
		\[
		\delta_p^2 = \Discr(p), \;\;\delta_q^2 = \Discr(q),\;\;\delta_r^2 = \Discr(r).
		\]
		Then the hyperplane in $\CC\PP(3)$, whose homogeneous coordinates are given by the row
		\begin{equation}\label{eqHatHyperplane}
			\begin{pmatrix}
				\Delta(p,q,r)^2\delta_q\delta_r & \Delta(p,q,r)^2\delta_r\delta_p & 
				\Delta(p,q,r)^2\delta_p\delta_q & 1
			\end{pmatrix}\mathfrak C_{p,q,r}^{-1}
		\end{equation}
		is invariant with respect to symmetries $\mathfrak X_{\hat N_p}$, $\mathfrak X_{\hat N_q}$, $\mathfrak X_{\hat N_r}$.
	\end{corollary}
	\begin{proof}
		From Proposition~\ref{propBasicProperties}~\ref{BasicPropertiesiii} and Lemma~\ref{lemSymmetriesFormula} it is evident that
		\[
		\mathfrak X_{\hat N_p} = \mathfrak C_{p,q,r} \begin{pmatrix}
			0 & 0 & 0 & 1 \\
			0 & 0 & \Delta(p,q,r)^2\delta_q^2 & 0 \\
			0 & \Delta(p,q,r)^2\delta_r^2 & 0 & 0 \\
			\Delta(p,q,r)^4 \delta_q^2\delta_r^2 & 0 & 0 & 0
		\end{pmatrix} \mathfrak C_{p,q,r}^{-1}.
		\]
		Now multiplying this matrix by the row~\eqref{eqHatHyperplane} yields
		\begin{multline*}
			\begin{pmatrix}
				\Delta(p,q,r)^2\delta_q\delta_r & \Delta(p,q,r)^2\delta_r\delta_p & 
				\Delta(p,q,r)^2\delta_p\delta_q & 1
			\end{pmatrix}\mathfrak C_{p,q,r}^{-1} \mathfrak X_{\hat N_p} = \\\Delta(p,q,r)^2\delta_q\delta_r\begin{pmatrix}
				\Delta(p,q,r)^2\delta_q\delta_r & \Delta(p,q,r)^2\delta_r\delta_p & 
				\Delta(p,q,r)^2\delta_p\delta_q & 1
			\end{pmatrix}\mathfrak C_{p,q,r}^{-1}.
		\end{multline*}
		Thus we have proved the invariance with respect to $\mathfrak X_{\hat N_p}$. Similar calculations show the invariance with respect to $\mathfrak X_{N_q}$ and $\mathfrak X_{N_r}$.
	\end{proof}
	
	That is,~\eqref{eqHatHyperplane} gives a formula for the rows of the matrix $\hat {\mathfrak B}$ from Lemma~\ref{lemMainKummerLemma} (indeed by changing the signs of the numbers $\delta_p,\delta_q,\delta_r$ we can obtain four linearly independent rows by~\eqref{eqHatHyperplane}). It remains to perform a similar calculation for the rows of the matrix $\mathfrak B$ from Lemma~\ref{lemMainKummerLemma}. However, it will be more convenient to calculate the columns of $\mathfrak B^{-1}$ instead. Clearly, each of the columns of $\mathfrak B^{-1}$ is a vector that belongs to the intersection of three out of four tropes from the property~\ref{mainKummeriii} of Lemma~\ref{lemMainKummerLemma}.
	
	To formulate the result of this calculation we introduce another matrix
	\begin{equation}\label{eqMatrixD}
		\mathfrak D_{p,q,r} = \begin{pmatrix}
			\hat p_2 & \hat q_2 & \hat r_2 & 0 \\
			-\hat p_1 & -\hat q_1 & -\hat r_1 & 0 \\
			-\hat p_0 & -\hat q_0 & -\hat r_0 & 0 \\
			\dfrac{\phi_1(p,q,r)}{8} &
			\dfrac{\phi_2(p,q,r)}{8} &
			\dfrac{\phi_3(p,q,r)}{8} & -\dfrac{\Delta(p,q,r)}{8}
		\end{pmatrix},
	\end{equation}
	where
	\[
	\begin{gathered}
		\phi_1(p,q,r) = (\Delta(p,q,r) - p_1\hat p_1)q_1r_1,\;\;\phi_2(p,q,r) = (\Delta(p,q,r) - q_1\hat q_1)r_1p_1,\\ \phi_3(p,q,r) = (\Delta(p,q,r) - r_1\hat r_1)p_1q_1.
	\end{gathered}
	\]
	It is easy to see that this matrix is invertible (again, since $\Delta(p,q,r) \ne 0$). Moreover, for $g,h \in \mathfrak P_1$ we introduce $\mathfrak r(g,h)$ as the determinant of the matrix
	\[
	\begin{pmatrix}
		g_0 & h_0 \\
		g_1 & h_1
	\end{pmatrix}.
	\]
	That is, $\mathfrak r(g,h)$ is the resultant of the degree $1$ polynomials $g,h$. It is straightforward to verify that $\mathfrak r(g,h)^2 = \Discr(gh)$ for all $g,h \in \mathfrak P_1$. With this notation we have the following result.
	
	\begin{lemma}\label{lemColumnsOfBinv}
		Let $\delta_p, \delta_q, \delta_r \in \CC$ be any numbers such that
		\[
		\delta_p^2 = \Discr(p), \;\;\delta_q^2 = \Discr(q),\;\;\delta_r^2 = \Discr(r).
		\]
		Then the vector
		\begin{equation}\label{eqVectorInThreeTropes}
			\mathfrak D_{p,q,r}\begin{pmatrix}
				\delta_p \\ \delta_q \\ \delta_r \\ \delta_p\delta_q\delta_r
			\end{pmatrix}
		\end{equation}
		belongs to the intersection of three out of four tropes described in the property~\ref{mainKummeriii} of Lemma~\ref{lemMainKummerLemma}.
	\end{lemma}
	\begin{proof}
		Consider the decompositions $p = p^{(1)}p^{(2)}$, $q = q^{(1)}q^{(2)}$, and $r = r^{(1)}r^{(2)}$ where all the factors belong to $\mathfrak P_1$. We choose the order of factors to satisfy the equalities
		\[
		\mathfrak r\left(p^{(1)}, p^{(2)}\right) = \delta_p,\;\;\mathfrak r\left(q^{(1)}, q^{(2)}\right) = \delta_q,\;\;\mathfrak r\left(r^{(1)}, r^{(2)}\right) = \delta_r.
		\]
		Finally consider all the Weierstrass points $P_1,P_2,Q_1,Q_2,R_1,R_2$ of $\mathcal X_f$ ordered in a way such that the $x$-coordinate of $P_j$ is the root of $p^{(j)}$, $j = 1,2$, etc. We claim that the vector~\eqref{eqVectorInThreeTropes} belongs to the intersection
		\[
		\mathfrak T_{P_1, Q_1, R_2} \cap \mathfrak T_{P_1, Q_2, R_1} \cap \mathfrak T_{P_2, Q_1, R_1}.
		\]
		Clearly, the symmetry of the formula~\eqref{eqVectorInThreeTropes} allows us to check only that 
		\[
		\mathfrak D_{p,q,r}\begin{pmatrix}
			\delta_p \\ \delta_q \\ \delta_r \\ \delta_p\delta_q\delta_r
		\end{pmatrix} \in \mathfrak T_{P_2, Q_1, R_1},
		\]
		which by~\eqref{eqTrope3} reduces to the equality
		\begin{equation}\label{eqTropeEqualityForMaple}
			\begin{pmatrix} g_2h_0 + g_0h_2 & g_3 h_0 + g_0 h_3 & -(g_3 h_1 + g_1 h_3) & 4\end{pmatrix}\; \mathfrak D_{p,q,r}\begin{pmatrix}
				\delta_p \\ \delta_q \\ \delta_r \\ \delta_p\delta_q\delta_r
			\end{pmatrix} = 0,
		\end{equation}
		where $g = p^{(2)}q^{(1)}r^{(1)}$ and $h = p^{(1)}q^{(2)}r^{(2)}$. This can be verified directly by substituting the definition~\eqref{eqMatrixD} of the matrix $\mathfrak D_{p,q,r}$ into~\eqref{eqTropeEqualityForMaple}, and observing that the result is a polynomial in terms of coefficients of $p^{(j)}$, $q^{(j)}$, and $r^{(j)}$. We have verified~\eqref{eqTropeEqualityForMaple} using Maple.
	\end{proof}  
	
	\subsection{Relation between Kleinian functions associated with Richelot isogenous curves}
	
	Now we are ready to formulate the final result of our calculations. For $p,q,r \in \mathfrak P_2$ we let $\mathscr T_{p,q,r}$ be the map from $\CC^4$ to itself defined by the formula
	\begin{equation}\label{eqTDef}
		\mathscr T_{p,q,r}(a) = \begin{pmatrix}
			2\Delta(p,q,r)^2a_1a_4 + 2\Delta(p,q,r)^4 \Discr(p)a_2a_3 \\
			2\Delta(p,q,r)^2a_2a_4 + 2\Delta(p,q,r)^4 \Discr(q)a_1a_3 \\
			2\Delta(p,q,r)^2a_3a_4 + 2\Delta(p,q,r)^4 \Discr(r)a_1a_2 \\
			\Delta(p,q,r)^4(\Discr(q)\Discr(r) a_1^2 + \Discr(p)\Discr(r) a_2^2 + \Discr(p)\Discr(q) a_3^2) + a_4^2
		\end{pmatrix}.
	\end{equation}
	\begin{theorem}\label{thMainKummerFormula}
		Assume that $f = pqr$ is admissible, where $p,q,r \in \mathfrak P_2$ and $\Delta(p,q,r) \ne 0$. Then the mapping 
		\begin{equation}\label{eqFormulaForR}
			a \mapsto \mathfrak D_{p,q,r}\mathscr T_{p,q,r}(\mathfrak C_{p,q,r}^{-1}a)
		\end{equation} 
		is a homogeneous of order $2$ map of $\CC^4$ onto itself that maps non-zero vectors to non-zero vectors. Moreover, the map $\mathfrak R$ induced by it on $\CC\PP^3$ makes the diagram~\eqref{eqMainPropertyR} commutative.
	\end{theorem}
	\begin{proof}
		We let $\hat p, \hat q, \hat r$ keep their usual meaning (see Subsection~\ref{subsecRichelotIsogeny}). Consider arbitrary numbers $\delta_p$, $\delta_q$, and $\delta_r$ such that \[
		\delta_p^2 = \Discr(p), \;\;\delta_q^2 = \Discr(q),\;\;\delta_r^2 = \Discr(r).
		\]
		Now consider
		\[
		\mathfrak U = \begin{pmatrix}
			\Delta(p,q,r)^2\delta_q\delta_r & \Delta(p,q,r)^2\delta_p\delta_r & \Delta(p,q,r)^2\delta_p\delta_q & 1\\
			-\Delta(p,q,r)^2\delta_q\delta_r & -\Delta(p,q,r)^2\delta_p\delta_r & \Delta(p,q,r)^2\delta_p\delta_q & 1\\
			-\Delta(p,q,r)^2\delta_q\delta_r & \Delta(p,q,r)^2\delta_p\delta_r & -\Delta(p,q,r)^2\delta_p\delta_q & 1\\
			\Delta(p,q,r)^2\delta_q\delta_r & -\Delta(p,q,r)^2\delta_p\delta_r & -\Delta(p,q,r)^2\delta_p\delta_q & 1
		\end{pmatrix}.
		\]        
		From Corollary~\ref{corColumnsOfHatB} it follows that the matrix $\hat {\mathfrak B} = \mathfrak U \mathfrak C_{p,q,r}^{-1}$ satisfies condition~\ref{mainKummerii} of Lemma~\ref{lemMainKummerLemma}. Thus from the remark that follows this lemma there exists a matrix $\mathfrak B$ that satisfies the conditions~\ref{mainKummeri} and~\ref{mainKummeriii}. From Lemma~\ref{lemFirstThreeRows} it follows that the first three rows of $\mathfrak B^{-1}$ can be modulo multiplication by a scalar factor calculated as first three rows of $\hat {\mathfrak B}^{-1}$ (note that all entries of the last column of $\hat {\mathfrak B}$ have the same value). A straightforward calculation shows that 
		\[
		\hat {\mathfrak B}^{-1} = \begin{pmatrix}
			\begin{pmatrix}
				\hat p_2 & \hat q_2 & \hat r_2 \\
				-\hat p_1 & -\hat q_1 & -\hat r_1 \\
				-\hat p_0 & -\hat q_0 & -\hat r_0
			\end{pmatrix} \begin{pmatrix}
				\delta_p & -\delta_p & -\delta_p & \delta_p \\
				\delta_q & -\delta_q & \delta_q & -\delta_q \\
				\delta_r & \delta_r & -\delta_r & -\delta_r
			\end{pmatrix} \\
			*
		\end{pmatrix},
		\]
		where $*$ denotes arbitrary row of $4$ numbers. On the other hand, the columns of $\mathfrak B^{-1}$ have the form given in~\eqref{eqVectorInThreeTropes}. Thus, modulo a scalar factor we have $\mathfrak B^{-1} = \mathfrak D_{p,q,r}\mathfrak V$, where
		\[
		\mathfrak V = \begin{pmatrix}
			\delta_p & -\delta_p & -\delta_p & \delta_p \\
			\delta_q & -\delta_q & \delta_q & -\delta_q \\
			\delta_r & \delta_r & -\delta_r & -\delta_r \\
			\delta_p\delta_q\delta_r & \delta_p\delta_q\delta_r & \delta_p\delta_q\delta_r & \delta_p\delta_q\delta_r
		\end{pmatrix}.
		\]
		The proof is finished by a calculation that shows that the map $a \mapsto \mathfrak V\mathrm{Sq}(\mathfrak Ua)$ coincides with $\mathscr T_{p,q,r}$ defined in~\eqref{eqTDef} (we also have verified this in Maple).
	\end{proof}

	Now we finalize our calculation. That is, we find an explicit formula that expresses $\mathscr S_f(z)$ through $\mathscr S_{\hat f}(z)$. This amounts to finding the matrix $\mathfrak H_{p,q,r}$ defined in Lemma~\ref{lemEtaTransformation}.
	
	\begin{definition}\label{defMatrices}
		For $p,q,r \in \mathfrak P_2$ such that $pqr$ is admissible and $\Delta(p,q,r) \ne 0$ we let $4 \times 4$ symmetric matrices $\mathfrak A_{p,q,r}$, $\mathfrak A^{(22)}_{p,q,r}$, $\mathfrak A^{(12)}_{p,q,r}$ and $\mathfrak A^{(11)}_{p,q,r}$ satisfy the equality
		\[
		\begin{pmatrix}
			a^T\mathfrak A_{p,q,r}a \\ a^T \mathfrak A^{(22)}_{p,q,r} a \\ a^T \mathfrak A^{(12)}_{p,q,r} a \\ a^T \mathfrak A^{(11)}_{p,q,r}a 
		\end{pmatrix} = \mathfrak D_{p,q,r}\mathscr T_{p,q,r}(\mathfrak C^{-1}_{p,q,r}a)
		\]
		for all $a \in \CC^4$.
	\end{definition}
	
	By Theorem~\ref{thMainKummerFormula} matrices from Definition~\ref{defMatrices} represent the Richelot isogeny $\mathscr R_{p,q,r}$. An explicit formula for these matrices can be obtained directly from definition (i.e. formula~\eqref{eqFormulaForR}), but we shall not give it here, as it is very cumbersome. However, we wish to compute the last row (or column) in the matrix $\mathfrak A^{(11)}_{p,q,r}$, since we can find the matrix $\mathfrak H^{p,q,r}$ from it by Lemma~\ref{lemMainTransformation}.
	
	To give an explicit formula for the matrix $\mathfrak H^{p,q,r}$ we introduce
	\begin{equation}\label{eqMuDef}
		\mu_{jklm}(p,q,r) = \hat p_j p_k p_lq_mr_m + \hat q_j q_k q_lp_mr_m + \hat r_j r_k r_lp_mq_m,
	\end{equation}
	where indices $j,k,l,m$ vary over the set $\{0,1,2\}$, $p,q,r \in \mathfrak P_2$ and as usual $\hat p = [q,r]$, $\hat q = [r,p]$ and $\hat r = [p,q]$. With~\eqref{eqMuDef} the matrix $\mathfrak C_{p,q,r}^{-1}$ can be expressed by the formula
	\begin{equation}\label{eqMatrixCinv}
		\mathfrak C_{p,q,r}^{-1} = \begin{pmatrix}
			-\dfrac{p_0}{\Delta(p,q,r)} & -\dfrac{p_1}{2\Delta(p,q,r)} & \dfrac{p_2}{\Delta(p,q,r)} & 0 \\[0.8em]
			-\dfrac{q_0}{\Delta(p,q,r)} & -\dfrac{q_1}{2\Delta(p,q,r)} & \dfrac{q_2}{\Delta(p,q,r)} & 0 \\[0.8em]
			-\dfrac{r_0}{\Delta(p,q,r)} & -\dfrac{r_1}{2\Delta(p,q,r)} & \dfrac{r_2}{\Delta(p,q,r)} & 0 \\[0.8em]
			\psi_0(p,q,r) &
			\psi_1(p,q,r) &
			\psi_2(p,q,r) & -16\Delta(p,q,r)
		\end{pmatrix},
	\end{equation}
	where 
	\[
	\begin{split}
		\psi_0(p,q,r) = & \; 4\mu_{0002}(p,q,r) + \mu_{2110}(p,q,r) + \mu_{2001}(p,q,r), \\
		\psi_1(p,q,r) = &  -\frac{\Delta(p,q,r)}{2}p_1q_1r_1 - \mu_{1021}(p,q,r), \\
		\psi_2(p,q,r) = & -4\mu_{2220}(p,q,r) - \mu_{0112}(p,q,r) - \mu_{0221}(p,q,r).
	\end{split}   
	\]
	Now by collecting all the formulas together we get
	\begin{multline}\label{eqA11lastcolumn}
		\mathfrak A^{(11)}_{p,q,r} = \\ \begin{pmatrix}
			* & * & * & 2\Delta(p,q,r)^3(p_0q_1r_1 + p_1q_0r_1 + p_1q_1r_0) + 2\Delta(p,q,r)^2\psi_0(p,q,r)\\
			* & * & * & -2\Delta(p,q,r)^2\mu_{1021}(p,q,r)\\
			* & * & * & -2\Delta(p,q,r)^3(p_2q_1r_1 + p_1q_2r_1 + p_1q_1r_2) + 2\Delta(p,q,r)^2\psi_2(p,q,r)\\
			* & * & * & -32\Delta(p,q,r)^3\\
		\end{pmatrix}.
	\end{multline}
	Both formulas~\eqref{eqMatrixCinv} and~\eqref{eqA11lastcolumn} were verified in Maple.
	
	Combining~\eqref{eqA11lastcolumn} with Lemma~\ref{lemMainTransformation} we obtain the following result.
	\begin{theorem}\label{thMainTransformation}
		Assume that $f = pqr$ is admissible, where $p,q,r \in \mathfrak P_2$ and $\Delta(p,q,r) \ne 0$. Let $\hat f$ be the result of the Richelot construction associated with the decomposition $f = pqr$. Then the equality
		\begin{equation*}\label{eqSMainTransform}
			\mathscr S_f(z) = -\frac{\exp\left(z^T\mathfrak H^{p,q,r}z\right)}{32\Delta(p,q,r)^3}\begin{pmatrix}
				\left(\mathscr S_{\hat f}(z)\right)^T \mathfrak A_{p,q,r}\mathscr S_{\hat f}(z) \\
				\left(\mathscr S_{\hat f}(z)\right)^T \mathfrak A_{p,q,r}^{(22)}\mathscr S_{\hat f}(z) \\
				\left(\mathscr S_{\hat f}(z)\right)^T \mathfrak A_{p,q,r}^{(12)}\mathscr S_{\hat f}(z) \\
				\left(\mathscr S_{\hat f}(z)\right)^T \mathfrak A^{(11)}_{p,q,r}\mathscr S_{\hat f}(z)
			\end{pmatrix}
		\end{equation*} holds, where
		\begin{multline*}\label{eqChiFormula}
			8\mathfrak H^{p,q,r} = \\ \begin{pmatrix}
				p_0q_1r_1 + p_1q_0r_1 + p_1q_1r_0 + \dfrac{\psi_0(p,q,r)}{\Delta(p,q,r)} & -\dfrac{\mu_{1021}(p,q,r)}{\Delta(p,q,r)} \\
				-\dfrac{\mu_{1021}(p,q,r)}{\Delta(p,q,r)} & p_2q_1r_1 + p_1q_2r_1 + p_1q_1r_2 - \dfrac{\psi_2(p,q,r)}{\Delta(p,q,r)}
			\end{pmatrix}
		\end{multline*}
	\end{theorem}

    \printbibliography
\end{document}